\documentclass[a4paper]{amsart}
\usepackage{amsmath}%
\usepackage{amsfonts}%
\usepackage{amssymb}%
\usepackage{graphicx}
\usepackage{amsthm}
\usepackage{amssymb}
\usepackage{tikz}
\usepackage{tikz-cd}
\usetikzlibrary{matrix,arrows}
\usepackage[english, ngerman]{babel}
\usepackage[utf8]{inputenc}
\usepackage{hyperref}
\usepackage{cite}
\usepackage{footnote}
\usepackage[a4paper,margin=3cm]{geometry}

\newtheorem{thm}{Theorem}
\theoremstyle{plain}
\newtheorem*{acknowledgement}{Acknowledgement}

\newtheorem{cor}{Corollary}

\newtheorem{defn}{Definition}

\newtheorem{lem}{Lemma}

\newtheorem{prop}{Proposition}
\newtheorem{rem}{Remark}

\numberwithin{equation}{section}

\renewcommand{\phi}{\varphi}

\newcommand{\BC}{{\mathbb{C}}}

\newcommand{\BF}{{\mathbb{F}}}

\newcommand{\BH}{{\mathbb{H}}}

\newcommand{\BN}{{\mathbb{N}}}

\newcommand{\BP}{{\mathbb{P}}}
\newcommand{\BQ}{{\mathbb{Q}}}

\newcommand{\BZ}{{\mathbb{Z}}}

\newcommand{\FS}{{\mathfrak{S}}}

\newcommand{\CM}{{\mathcal M}}

\renewcommand{\mod}{\mathop{\rm mod}\nolimits}
\newcommand{\sign}{\mathop{\rm sign}\nolimits}

\newcommand{\im}{\mathop{{\rm Im}}\nolimits}

\newcommand{\GL}[1]{\mathop{\rm GL}_{#1} \nolimits}
\newcommand{\SL}[1]{\mathop{\rm SL}_{#1} \nolimits}

\newcommand{\legendre}[2]{\genfrac(){}{}{#1}{#2}}
\newcommand{\slegendre}[2]{\genfrac(){}{1}{#1}{#2}}

\newcommand{\sym}{\mathop{\rm sym}}

\newcommand{\Log}{\mathop{\rm Log}\nolimits}

\newcommand{\quotient}[2]{
        \mathchoice
            {
                \text{\raise1ex\hbox{$#1$}\Big/\lower1ex\hbox{$#2$}}%
            }
            {
                #1\,/\,#2
            }
            {
                #1\,/\,#2
            }
            {
                #1\,/\,#2
            }
    }
		
\newcommand{\rquotient}[2]{
        \mathchoice
            {
                \text{\lower1ex\hbox{$#1$}\Big \backslash \raise01ex\hbox{$#2$}}%
            }
            {
                #1\,\backslash\,#2
            }
            {
                #1\,\backslash\,#2
            }
            {
                #1\,\backslash\,#2
            }
    }
		
\newcommand{\lrquotient}[3]{
        \mathchoice
            {
                \text{\lower1ex\hbox{$#1$}\Big \backslash \raise01ex\hbox{$#2$}\Big/\lower1ex\hbox{$#3$}}%
            }
            {
                #1\,\backslash\,#2\,/\,#3
            }
            {
                #1\,\backslash\,#2\,/\,#3
            }
            {
                #1\,\backslash\,#2\,/\,#3
            }
    }

\begin{document}
\selectlanguage{english}

\bibliographystyle{alpha}

\title[Supnorm of Modular Forms of half-integral Weight]{Supnorm of Modular Forms of half-integral Weight in the Weight Aspect}
\author{Raphael S. Steiner}
\address{Department of Mathematics, University of Bristol, Bristol BS8 1TW, UK}%
\email{raphael.steiner@bristol.ac.uk}%




\begin{abstract} We bound the supnorm of half-integral weight Hecke eigenforms in the Kohnen plus space of level $4$ in the weight aspect, by combining bounds obtained from the Fourier expansion with the amplification method using a Bergman kernel.
\end{abstract}
\maketitle

\section{Introduction}

The question of supremum norms of holomorphic and Maass Hecke eigenforms are connected to $L$-functions attached to them. In the case of holomorphic half-integral weight Hecke eigenforms they are directly related to the critical values of quadratic twists of the $L$-functions associated to their Shimura lift. Therefore supnorms have been studied by many in various ways: Iwaniec-Sarnak \cite{IS95} in the eigenvalue aspect, Harcos-Templier \cite{HT2}, \cite{HT3} and Saha \cite{Saha14} in the level aspect, as well as Kiral \cite{halflevel} in the case of half-integral weight, Templier \cite{Thybrid} in the level as well as the eigenvalue aspect, unifying both best known results. In the weight aspect they have been studied by Xia \cite{Supnormintweight}, Das-Sengupta \cite{DasSeng}, Rudnick \cite{R05}, Friedman-Jorgenson-Kramer \cite{FJK} and the author himself \cite{realweight}, where in the last three the condition of being a Hecke eigenform is not necessary.\\

In this paper we are concerned with the supremum norm in the weight aspect of holomorphic half-integral weight Hecke eigenforms in the Kohnen plus space of level 4. Assuming the Lindel\"of hypothesis we are able to prove an analogue of Xia's result \cite{Supnormintweight}, which states that for a for a holomorphic Hecke eigenform $f$ of integral weight for the full modular group $\SL2(\BZ)$ we have $\sup_{z \in \BH} y^{\frac{k}{2}} |f(z)| \ll_{\epsilon} k^{\frac{1}{4}+\epsilon}$. Our theorem reads as follows.

\begin{thm} Let $k \in \frac{1}{2}+\BZ$ with $k \ge \frac{5}{2}$ and $f \in S^+_k(\Gamma_0(4)^{\star})$ be a $L^2$-normalised Hecke eigenform ($\langle f,f \rangle_{\Gamma_0(4)}=1$) of half-integral weight $k$ contained in the Kohnen plus space. Assume the Lindel\"of hypothesis for the family of $L$-functions $L(F,\chi,s)$, where $F$ is any modular form of weight $2k-1$ on $\SL2(\BZ)$ and $\chi$ any primitive quadratic character. Then we have
$$
\sup_{z \in \BH} y^{\frac{k}{2}}|f(z)| \ll_{\epsilon} k^{\frac{1}{4}+\epsilon}.
$$
\label{thm:2}
\end{thm}

Unconditionally we are able to prove the following.

\begin{thm} Let $k \in \frac{1}{2}+\BZ$ with $k \ge \frac{5}{2}$ and $f \in S^+_k(\Gamma_0(4)^{\star})$ be a $L^2$-normalised Hecke eigenform ($\langle f,f \rangle_{\Gamma_0(4)}=1$) of half-integral weight $k$ contained in the Kohnen plus space. Then we have
$$
\sup_{z \in \BH} y^{\frac{k}{2}}|f(z)| \ll_{\epsilon} k^{\frac{3}{7}+\epsilon}.
$$
\label{thm:1}
\end{thm}

We now give a brief overview of the significance of the various exponents and the methods which go into them. If $f$ is not assumed to be a Hecke eigenform, then the best exponent one can prove in general is $k^{\frac{3}{4}}$. Indeed this has been shown for arbitrary real weight $k$ by the author \cite{realweight} and relies on estimates for the Fourier coefficients of Poincar\'e series. However, when $f$ is an eigenform of half-integral weight as in the current paper, it follows from a result of Kohnen and Zagier (or more generally Waldspurger) that the square of its Fourier coefficients are essentially central $L$-values. Using the convexity bound on said $L$-functions one achieves a bound for the Fourier expansion, which is especially good near the cusps. Combing this estimate with a Bergman kernel for the case away from the cusps gives the bound $k^{\frac{1}{2}+\epsilon}$ for the supnorm. Any sub-convexity result on those central $L$-values easily allows the removal of the $\epsilon$ to achieve the bound $k^{\frac{1}{2}}$; this was shown by the author in his master's thesis. To decrease the exponent further one can either use deeper techniques or one can assume unproven bounds, e.g. the Lindel\"of hypothesis as in Theorem \ref{thm:2}. The bound $k^{\frac{1}{4}+\epsilon}$ is essentially best possible as the next theorem shows that the best uniform bound one can hope for is $k^{\frac{1}{4}}$, if one takes the dimension of the space into consideration. The bound $k^{\frac{3}{7}}$ comes from combining the best known bound for these central $L$-values given by Petrow \cite{Pet14} and Young \cite{Young14} combined with the amplification method using the Bergman kernel.

\begin{thm}Let $k \in \frac{1}{2}+\BZ$ with $k \ge \frac{5}{2}$ and $\{f_j\} \subseteq S^+_k(\Gamma_0(4)^{\star})$ be an orthonormal basis of Hecke eigenforms of half-integral weight $k$ contained in the Kohnen plus space. Let $\{F_j\} \subseteq S_{2k-1}(\SL2(\BZ))$ be the corresponding arithmetically normalised Hecke eigenforms ($\widehat{F_j}(1)=1$) under the Shimura map. Then we have the following lower bounds:
$$
\sup_{z \in \BH} y^{\frac{k}{2}}|f_j(z)| \gg_{\epsilon} \max\left\{1, \ k^{\frac{1}{4}-\epsilon} \sup_{\substack{D \text{ fund. disc.},\\ (-1)^{k-\frac{1}{2}}D>0}} L(F_j,\slegendre{D}{\cdot},1/2)^{\frac{1}{2}} |D|^{-\frac{1}{2}}\right\},
$$
\vspace{4mm}
$$
\sum_j \sup_{z \in \BH} y^k |f_j(z)|^2 \ge \sup_{z \in \BH} \sum_j y^k |f_j(z)|^2 \gg k^{\frac{3}{2}}.
$$
\label{thm:3}
\end{thm}

Although we restrict ourselves in this paper to the Kohnen plus space of level 4, the methods certainly generalise to larger level, but slightly weaker results are to be expected. Nevertheless the author strongly believes that even the convexity bound on the critical value of the corresponding $L$-functions are sufficient to break the convexity bound of $k^{\frac{3}{4}}$ as this is indeed the case in the Kohnen plus space of level 4.
\newpage
\section{Notation and Preliminaries} Throughout let $k\in \frac{1}{2}+\BZ$ be a half-integer with $k \ge \frac{5}{2}$. For a complex number $z\in \BC^{\times}$ we define $z^k=\exp(k\cdot \Log(z))$, where $\Log(z)=\log|z|+i\arg(z)$ with $-\pi < \arg(z)\le \pi$. The notation $f(x) \ll_{A,B} g(x)$ means that $|f(x)|\le K g(x)$, where $K$ is some function depending at most on $A$ and $B$. Further let $e(z)=\exp(2 \pi i z)$ for $z \in \BC$.\\

As usual we define the M\"obius action of $\gamma \in \GL2^{+}(\BQ)$, the set of all $2\times 2$ matrices with rational coefficients and positive determinant, on $\BH$, the upper half plane, as
$$
\gamma \cdot z= \gamma z = \frac{az+b}{cz+d}, \quad \forall \gamma = \begin{pmatrix} a & b \\ c & d \end{pmatrix} \in \GL2^{+}(\BQ), \forall z \in \BH.
$$
The action is extended to the set of cusps $\overline{\BQ}=\BP^{1}(\BQ)=\BQ \sqcup \{\infty\}$. We further define
$$
j(\gamma,z)=cz+d, \quad \forall \gamma= \begin{pmatrix} a & b \\ c & d \end{pmatrix} \in \GL2^{+}(\BQ), \forall z \in \BH
$$
and
$$
j_{\Theta}(\gamma,z)=\frac{\Theta(\gamma z)}{\Theta(z)}, \quad \forall \gamma= \begin{pmatrix} a & b \\ c & d \end{pmatrix} \in \Gamma_0(4), \forall z \in \BH,
$$
where $\Theta(z)=\sum_{n\in \BZ}e(n^2 z)$. By $\FS_k$ we denote the group, whose elements are of the form $(\gamma, \phi)$, where $\gamma \in \GL2^{+}(\BQ)$ and $\phi:\BH \to \BC$ a holomorphic function with $|\phi(z)|=(\det \gamma)^{-\frac{k}{2}} |j(\gamma,z)|^k$, and whose composition is given by:
$$
(\gamma,\phi)\circ(\gamma^{'},\phi^{'})=(\gamma \gamma^{'}, (\phi \circ \gamma^{'})\cdot \phi^{'}).
$$
For each $k$ we have a group homomorphism
$$\begin{aligned}
^{\star}: \Gamma_0(4)& \to \FS_k \\
\gamma & \mapsto \gamma^{\star}=(\gamma,j_{\Theta}(\gamma,\cdot)^{2k}).
\end{aligned}$$
We further have an inclusion as sets $\GL2^+(\BQ) \hookrightarrow \FS_k$, where we identify the element $\gamma \in \GL2^+(\BQ)$ with $(\gamma,(\det \gamma)^{-\frac{k}{2}} j(\gamma,\cdot)^k)$. Among all elements in $\FS_k$ we would like to distinguish two special elements $W_4$ and $V_4$, which we are going to use to translate the cusps $0,\frac{1}{2}$ to $\infty$,
$$\begin{aligned}
W_4 &= \left( \begin{pmatrix} 0 & -\frac{1}{2} \\ 2 & 0 \end{pmatrix}, (-2iz)^k \right),\\
V_4 &= \left( \begin{pmatrix} 1 & 0 \\ 2 & 1 \end{pmatrix}, (-i(2z+1))^k \right).
\end{aligned}$$

\begin{defn} For $\tau \in \SL2(\BZ)$ we define the \emph{cusp width} $n_{\tau}$ and the \emph{cusp parameter} $\kappa_{\tau} \in [0,1)$ in such a way that the stabilizer group at $\infty$ of $\tau^{-1}\Gamma_0(4)^{\star}\tau$ is generated by
$$
\pm \left(\begin{pmatrix} 1 & n_{\tau} \\ 0 & 1 \end{pmatrix}, e(\kappa_{\tau})  \right).
$$
\end{defn}

\begin{rem} For $\Gamma_0(4)^{\star}$, the cusps $0,\frac{1}{2},\infty$ have cusp width $4,1,1$ and cusp parameter $0,\frac{1}{2}-(-1)^{k-\frac{1}{2}}\frac{1}{4},0$ respectively.
\end{rem}

The group $\FS_k$ acts on the set of meromorphic functions on $\BH$ as follows:
$$
(f|_k (\gamma, \phi))(z)=\phi(z)^{-1} f(\gamma z).
$$
\begin{defn} A holomorphic function $f$ on the upper half-plane satisfying
$$
f|_k \xi=f, \quad \forall \xi \in \Gamma_0(4)^{\star},
$$
and having a Fourier expansion of the form
$$
(f|_k \tau)(z)=\sum_{m+\kappa_{\tau}>0} \widehat{(f|_k\tau)}(m) \, e\!\left(\frac{m+\kappa_{\tau}}{n_{\tau}}z\right)
$$
for every $\tau \in \SL2(\BZ)$ is called a \emph{cusp form} of weight $k$ with respect to $\Gamma_0(4)^{\star}$. The set of such functions we denote by $S_k(\Gamma_0(4)^{\star})$.
\end{defn}

The space $S_k(\Gamma_0(4)^{\star})$ is finite dimensional and can be made into a Hilbert space by defining the Petersson inner product:
$$
\langle f,g \rangle_{\Gamma_0(4)} = \frac{1}{6} \int_{\BF_{\Gamma_0(4)}} f(z)\overline{g(z)}y^k \frac{dxdy}{y^2},
$$
where $\BF_{\Gamma_0(4)}$ is a fundamental domain for $\Gamma_0(4)$ and $z=x+iy$. Furthermore a theory of Hecke operators can be established on $S_k(\Gamma_0(4)^{\star})$. For $l$ a square, one defines
$$
f|_k T(l)= l^{\frac{k}{2}-1} \sum_{\xi \in \rquotient{\Gamma_0(4)^{\star}}{\Gamma_0(4)^{\star}\xi_l \Gamma_0(4)^{\star}}} f|_k \xi,
$$
where
$$
\xi_l= \left( \begin{pmatrix} 1 & 0 \\ 0 & l \end{pmatrix}, l^{\frac{k}{2}} \right).
$$
These operators commute and thus one gets an orthonormal basis of Hecke eigenforms. Shimura \cite{MFhiw} has shown, that given such a Hecke eigenform $f$ one can use its Fourier coeficients to construct a classical Hecke eigenform $F \in S_{2k-1}(\Gamma_0(M))$ of weight $2k-1$ for some level $M$ with the same Hecke eigenvalues. Later Niwa \cite{MFohiwatIocTF} has shown that one can always take $M=2$, moreover Kohnen \cite{Ko80} has shown one can take $M=1$ if the eigenform is coming from a certain subspace, the Kohnen plus space, which is defined as follows:
$$
S^+_k(\Gamma_0(4)^{\star})= \{f \in S_k(\Gamma_0(4)^{\star}) | \widehat{f}(m)=0, \, \forall m \text{ such that } (-1)^{k-\frac{1}{2}}n \equiv 2,3 \mod(4) \}.
$$

The plus space has some nice properties, one of which is that it comes with a projection $S_k(\Gamma_0(4)^{\star}) \to S^+_k(\Gamma_0(4)^{\star})$. For this reason the subspace has its own Poincar\'e series, which have been computed by Kohnen \cite{Ko85}.

\begin{prop} Let $k\in \frac{1}{2}+ \BZ$ with $k \ge \frac{5}{2}$ and $m \in \BN, (-1)^{k-\frac{1}{2}}m \equiv 0,1 \mod(4)$. The Poincar\'e series $G_I^+( \Gamma_0(4)^{\star},k,z,m )$ given by the Fourier expansion:
$$
G_I^+( \Gamma_0(4)^{\star},k,z,m ) = \sum_{\substack{n \ge 1, \\ (-1)^{k-\frac{1}{2}}n \equiv 0,1 \mod(4)}}g_{k,m}(n)e^{2 \pi i n z},
$$
with
$$
g_{k,m}(n)=\frac{2}{3} \left[\delta_{m,n}+(-1)^{ \left\lfloor \frac{k+\frac{1}{2}}{2} \right\rfloor} \pi \sqrt{2} \left( \frac{n}{m}\right)^{\frac{k-1}{2}} \sum_{c \ge 1}H_c(n,m)J_{k-1}\left( \frac{\pi}{c} \sqrt{nm} \right) \right],
$$
where $H_c(n,m)$ is given by
$$\begin{aligned}
H_c(n,m) &=(1-(-1)^{k-\frac{1}{2}}i) \left(1+\legendre{4}{c} \right)\frac{1}{4c} \sum_{\substack{\delta \mod(4c),\\ (\delta,4c)=1}}\legendre{4c}{\delta}\legendre{-4}{\delta}^{k} e\left(\frac{n\delta+m\delta^{-1}}{4c} \right),
\end{aligned}$$
satisfy
$$
\langle f, G_I^+(\Gamma_0(4)^{\star},k, \cdot ,m) \rangle_{\Gamma_0(4)} = \frac{\Gamma(k-1)}{6\cdot(4 \pi m)^{k-1}} \widehat{f}(m) \quad \forall f \in S^+_k(\Gamma_0(4)^{\star})
$$
and $G_I^+(\Gamma_0(4)^{\star},k, \cdot ,m) \in S^+_k(\Gamma_0(4)^{\star})$.
\end{prop}
\begin{proof} See Proposition 4 of \cite{Ko85}.
\end{proof}
The following Corollary is immediate.

\begin{cor} Let $k \in \frac{1}{2}+\BZ$ with $k \ge \frac{5}{2}$ and $\{f_j\}$ be an orthonormal basis of $S^+_k(\Gamma_0(4)^{\star})$, then we have
$$
\sum_j \left|\widehat{f_j}(m)\right|^2=\frac{6 \cdot (4\pi m)^{k-1}}{\Gamma(k-1)} \cdot \frac{2}{3} \left[1+(-1)^{ \left\lfloor \frac{k+\frac{1}{2}}{2} \right\rfloor} \pi \sqrt{2} \sum_{c \ge 1}H_c(m,m)J_{k-1}\left( \frac{\pi m}{c} \right) \right].
$$
\label{cor:halfpluscoeff}
\end{cor}

Furthermore cusp forms in the Kohnen plus space have special relations among their Fourier coefficients at different cusps as the next lemma shows.

\begin{lem} Let $k \in \frac{1}{2}+\BZ, f \in S^+_k(\Gamma_0(4)^{\star})$. Then the Fourier coefficients of $f$ at the cusps $0,\frac{1}{2}$ can be given in terms of the Fourier coefficients at $\infty$:
$$\begin{aligned}
(f|_k W_4)(z) &= \legendre{2}{2k}2^{\frac{1}{2}-k} \sum_{m \ge 1} \widehat{f}(4m)e(mz), \\
(f|_k V_4)(z) &= \legendre{2}{2k}2^{\frac{1}{2}-k} \sum_{\substack{m \ge 1, \\ (-1)^{k-\frac{1}{2}}m \equiv 1 \mod (4)}} i^{\frac{m}{2}}\widehat{f}(m)e\left(\frac{m}{4}z\right).
\end{aligned}$$
We note here that $\legendre{2}{2k}$ denotes the Jacobi symbol.
\label{lem:Fourrierpluscusp}
\end{lem}

\begin{proof} In \cite{Ko80} Prop. 2 Kohnen showed: $(f| U_4 |_k W_4)(z)= \legendre{2}{2k} 2^{k-\frac{1}{2}} f(z)$, where $(f|U_4)(z)=\sum_{m\ge 1} \widehat{f}(4m)e(mz)$.
Applying $|_k W_4$ to both sides gives the desired result, by noting that $|_k W_4^2$ is the identity map. The second identity follows from:
$$ \begin{aligned}
(f|_k V_4)(z) &= (-i(2z+1))^{-k} \sum_{\substack{m \ge 1,\\ (-1)^{k-\frac{1}{2}}m \equiv 0,1 \mod(4)}}\widehat{f}(m) e\left(\frac{m}{2}-\frac{m}{4z+2}\right) \\
&= (-i(2z+1))^{-k} \left[ 2 \sum_{\substack{m \ge 1,\\ m \equiv 0 \mod(4)}} -\sum_{\substack{m \ge 1,\\ (-1)^{k-\frac{1}{2}}m \equiv 0,1 \mod(4)}} \right] \widehat{f}(m) e\left(\frac{-m}{4z+2} \right) \\
&= (-i(2z+1))^{-k} 2(f|U_4)\left( \frac{-1}{z+\frac{1}{2}}\right) - (f|_k W_4)\left(z+\frac{1}{2}\right) \\
&= 4^{\frac{1}{2}-k}(f| U_4 |_k W_4)\left(\frac{z+\frac{1}{2}}{4} \right) - (f|_k W_4)\left(z+\frac{1}{2}\right) \\
&= \legendre{2}{2k}2^{\frac{1}{2}-k} \left[ f\left( \frac{z+\frac{1}{2}}{4} \right)- (f|U_4)\left(z+\frac{1}{2}\right) \right] \\
&= \legendre{2}{2k} 2^{\frac{1}{2}-k} \sum_{\substack{m \ge 1,\\ (-1)^{k-\frac{1}{2}}m \equiv 1 \mod (4)}} i^{\frac{m}{2}} \widehat{f}(m) e\left(\frac{m}{4}z\right).
\end{aligned}$$
\end{proof}

If we now assume $f \in S^+_k(\Gamma_0(4)^{\star})$ to be a Hecke eigenform, we can even say more about its Fourier coefficients. In this case Waldspurger has shown, that the square of the Fourier coefficients are proportional to the central value of a certain twist of the $L$-function associated to its Shimura lift. We only need a special case, which has been made explicit by Kohnen-Zagier.

\begin{prop} Let $k \in \frac{1}{2}+ \BZ$ with $k \ge \frac{5}{2}$, $f \in S^+_k(\Gamma_0(4)^{\star})$ a Hecke eigenform and let $F \in S_{2k-1}(\SL2(\BZ))$ be the corresponding arithmetically normalised Hecke eigenform ($\widehat{F}(1)=1$) of $f$ under the Shimura map. Further let $D$ be a fundamental discriminant with $(-1)^{k-\frac{1}{2}}D>0$ and $L(F,\legendre{D}{\cdot},s)$ the analytic continuation of the Dirichlet L-series $\sum_{n=1}^{\infty}\legendre{D}{n} \frac{\hat{F}(n)}{n^{k-1}}n^{-s}$. Then
$$
\frac{|\hat{f}(|D|)|^2}{\langle f,f \rangle_{\Gamma_0(4)}} = \frac{\Gamma(k-\frac{1}{2})}{\pi^{k-\frac{1}{2}}} |D|^{k-1} \frac{L(F,\legendre{D}{\cdot},\frac{1}{2})}{\langle F,F \rangle_{\SL2(\BZ)}},
$$
where
$$
\langle F,F \rangle_{\SL2(\BZ)}=\int_{\BF_{\SL2(\BZ)}} |F(z)|^2y^{2k-1}\frac{dxdy}{y^2}
$$
and $\BF_{\SL2(\BZ)}$ is a fundamental domain of $\SL2(\BZ)$.
\label{prop:coeffscritvalue}
\end{prop}

\begin{proof} We refer to \cite{KZ81}.\end{proof}

Concerning the size of $\langle F,F \rangle_{\SL2(\BZ)}$ we have the following two propositions.

\begin{prop} \label{prop:norm} Let $F \in S_{2k-1}(\SL2(\BZ))$ be an arithmetically normalised Hecke eigenform, then we have:
$$
\langle F,F \rangle_{\SL2(\BZ)} = \frac{\Gamma(2k-1)}{2^{4k-3}\pi^{2k}} L( {\sym}^2 F,1),
$$
where $L(\sym^2 F,s)$ is the analytic continuation of
$$\begin{aligned}
\prod_p \left( 1- \alpha_p^2 p^{2-2k-s} \right)^{-1} &\left( 1-\alpha_p \overline{\alpha_p} p^{2-2k-s}  \right)^{-1} \left( 1- \overline{\alpha_p}^2 p^{2-2k-s} \right)^{-1}\\
=&\frac{\zeta(2s)}{\zeta(s)} \sum_{n=1}^{\infty} \frac{\widehat{F}(n)^2}{n^{s+2k-2}}
\end{aligned}$$
and $\alpha_p,\overline{\alpha_p}$ are the solutions to $\alpha_p+\overline{\alpha_p}=\widehat{F}(p), \, \alpha_p\overline{\alpha_p}=p^{2k-2}$.
\end{prop}
\begin{proof} See \cite{ContrRamfunc}.
\end{proof}

\begin{prop} \label{prop:sym} Let $F \in S_{2k-1}(\SL2(\BZ))$ be an arithmetically normalised Hecke eigenform, then we have:
$$
k^{-\epsilon} \ll_{\epsilon} L({\sym}^2 F,1) \ll_{\epsilon} k^{\epsilon}.
$$
\end{prop}
\begin{proof} See page 41 equation 2.16 of \cite{ANTALfunc}.
\end{proof}

If we adopt the notation of Proposition \ref{prop:coeffscritvalue} all the remaining Fourier coefficients of our Hecke eigenform $f\in S^+_k(\Gamma_0(4)^{\star})$ satisfy the following equation
\begin{equation}
\widehat{f}(n^2|D|)=\widehat{f}(|D|) \sum_{d|n}\mu(d) \legendre{D}{d}d^{k-\frac{3}{2}}\widehat{F}\left( \frac{n}{d} \right).
\label{eq:sqrcoeff}
\end{equation}

\section{Proof of Theorems}

Let $f=f_1\in S^+_k(\Gamma_0(4)^{\star})$ be a Hecke eigenform of norm $\langle f,f \rangle_{\Gamma_0(4)}=1$. Then $y^{\frac{k}{2}}|f(z)|$ is $\Gamma_0(4)$ invariant. Moreover we have that $(\im \tau z)^{\frac{k}{2}}|f(\tau z)|=y^{\frac{k}{2}}|(f|_k \tau)(z)|$ holds for all $\tau \in \GL2^+(\BQ)$. This and the fact that the set
$$
\left\{z|\im z \ge  \frac{\sqrt{3}}{8}\right\} \cup \left\{W_4 z|\im z \ge  \frac{\sqrt{3}}{8}\right\} \cup \left\{V_4 z|\im z \ge  \frac{\sqrt{3}}{8}\right\} 
$$
covers a fundamental domain of $\Gamma_0(4)$ imply the following equality
\begin{equation}
\sup_{z \in \BH} y^{\frac{k}{2}} |f(z)|= \max_{\xi \in \{I,W_4,V_4\}} \sup_{y \ge \frac{\sqrt{3}}{8}} y^{\frac{k}{2}} |(f|_k \xi)(z)|.
\label{eq:reducing}
\end{equation}

The proof of Theorem \ref{thm:2} and \ref{thm:1} is split up into two parts. In the first part we use the Fourier expansion and bounds on the Fourier coefficients to bound the supnorm near a cusp. If we are far away from the cusp we can use the Bergman kernel in combination with an amplifier to get superior results, which is described in the second part. In a third part we give the proof of Theorem \ref{thm:3}.

\subsection{Bounding the Fourier expansion}On a first thought it is tempting to use classical estimates such as
$$
\sum_{n\le N} |\widehat{f}(n)|^2 \ll_f N^k
$$
to bound the Fourier expansion, but it turns out that the implied constant is heavily dependent on $f$, in fact the supnorm of $f$ itself appears as a factor. Thus one might try and use deeper techniques or one can use the currently best known result towards the Ramanujan-Petersson conjecture. We follow the latter path.\\

Throughout we assume we have a uniform bound of the shape
\begin{equation}
L\left(F,\chi,\frac{1}{2}\right) \ll k^{\alpha} q^{\beta}
\label{eq:uniformbound}
\end{equation}
for all arithmetically normalised Hecke eigenforms $F \in S_{2k-1}(\SL2(\BZ))$ and quadratic characters $\chi$ of conductor $q$. Through the work of Petrow \cite{Pet14} and Young \cite{Young14} we now know that the pair $(\alpha,\beta)=(\frac{1}{3}+\epsilon,\frac{1}{3}+\epsilon)$ is permissible for all $\epsilon>0$. The Lindel\"of hypothesis corresponds of course to the pair $(\alpha,\beta)=(\epsilon,\epsilon)$.\\

Using Deligne's bound for the Fourier coefficients of $F\in S_{2k-1}(\SL2(\BZ))$ in equation \eqref{eq:sqrcoeff} we find that:

\begin{equation}
|\widehat{f}(n^2|D|)| \ll_{\epsilon} |\widehat{f}(|D|)| \cdot \sum_{d|n} d^{k-\frac{3}{2}} \left( \frac{n}{d} \right)^{k-1+\epsilon} \ll_{\epsilon} |\widehat{f}(|D|)| \cdot (n^2)^{\frac{k-1}{2}+\epsilon}.
\label{eq:sqrbound}
\end{equation}

Combining the Propositions \ref{prop:coeffscritvalue}, \ref{prop:norm} and \ref{prop:sym} with the bound \eqref{eq:uniformbound} we get:

\begin{equation}
|\widehat{f}(|D|)| \ll_{\epsilon} \frac{(4\pi)^{\frac{k}{2}}}{\Gamma(k)^{\frac{1}{2}}} \cdot |D|^{\frac{k-1+\beta}{2}} k^{\frac{\alpha}{2}+\epsilon}.
\label{eq:discbound}
\end{equation}

Thus we conclude the following proposition.

\begin{prop} Let $k \in \frac{1}{2}+\BZ$ with $k\ge \frac{5}{2}$ and $f \in S_k^+(\Gamma_0(4)^{\star})$ a $L^2$-normalised Hecke eigenform. Further assume we have a uniform bound as in \eqref{eq:uniformbound} with $\beta>0$, then we have the following estimate on its Fourier coefficients:
$$
|\widehat{f}(m)| \ll_{\epsilon} \frac{(4\pi)^{\frac{k}{2}}k^{\frac{\alpha}{2}+\epsilon}}{\Gamma(k)^{\frac{1}{2}}} \cdot m^{\frac{k-1+\beta}{2}}.
$$
\label{prop:fouriercoefbound}
\end{prop}

For convenience let us introduce the sum
\begin{equation}
S(\alpha,\beta,\kappa)=\sum_{m+\kappa>0} (m+\kappa)^{\alpha}e^{-\beta(m+\kappa)}, \quad \alpha,\beta,\kappa>0.
\label{eq:Sab}
\end{equation}

We will further need two lemmata for this sum.

\begin{lem} \label{lem:Sabest} $S(\alpha,\beta,\kappa)$ as defined by \eqref{eq:Sab} satisfies the following inequalities:
$$
S(\alpha,\beta,\kappa) \le \beta^{-\alpha-1}\Gamma(\alpha+1)+\beta^{-\alpha} \alpha^{\alpha} e^{-\alpha}
$$
and for $\alpha \le \beta \kappa$ we have:
$$
S(\alpha,\beta,\kappa) \le \beta^{-\alpha-1}\Gamma(\alpha+1) + \kappa^{\alpha} e^{-\beta\kappa}.
$$
\end{lem}

\begin{proof} This is Lemma 1 of \cite{realweight}.
\end{proof} 

\begin{lem} \label{lem:expdecay} The following inequality holds for $\kappa \ge 6 \frac{\alpha}{\beta}, \ \alpha,\beta>0$:
$$
\kappa^{\alpha}e^{-\beta \kappa} \le \alpha^{\alpha}\beta^{-\alpha}e^{-\alpha} \cdot e^{-\frac{\beta \kappa}{2}}.
$$
\end{lem}

\begin{proof} This is Lemma 2 of \cite{realweight}.
\end{proof}

Using Proposition \ref{prop:fouriercoefbound} in Lemma \ref{lem:Fourrierpluscusp} we find:

\begin{equation}\begin{aligned}
y^{\frac{k}{2}}|f(z)| & \ll_{\epsilon} \frac{y^{\frac{k}{2}}(4 \pi)^{\frac{k}{2}}k^{\frac{\alpha}{2}+\epsilon}}{\Gamma(k)^{\frac{1}{2}}} S \left(\frac{k-1+\beta}{2},2 \pi y,1 \right) , \\
y^{\frac{k}{2}}|(f|_k W_4)(z)| & \ll_{\epsilon} \frac{y^{\frac{k}{2}}(4 \pi)^{\frac{k}{2}}k^{\frac{\alpha}{2}+\epsilon}}{\Gamma(k)^{\frac{1}{2}}} S \left(\frac{k-1+\beta}{2},2 \pi y,1 \right), \\
y^{\frac{k}{2}}|(f|_k V_4)(z)| & \ll_{\epsilon} \frac{\left(\frac{y}{4}\right)^{\frac{k}{2}}(4 \pi)^{\frac{k}{2}}k^{\frac{\alpha}{2}+\epsilon}}{\Gamma(k)^{\frac{1}{2}}} S \left(\frac{k-1+\beta}{2},\frac{2 \pi y}{4}, 1 \right).
\label{eq:halfintysmall}
\end{aligned}\end{equation}

Using Lemma \ref{lem:Sabest} we find:
$$\begin{aligned}
S \left(\frac{k-1+\beta}{2},2 \pi y,1 \right) & \ll (4 \pi)^{-\frac{k}{2}+\frac{1}{2}-\frac{\beta}{2}} y^{-\frac{k}{2}-\frac{1}{2}-\frac{\beta}{2}} k^{\frac{k}{2}+\frac{\beta}{2}}e^{-\frac{k}{2}}  \left(1+yk^{-\frac{1}{2}} \right).
\end{aligned}$$

Thus we get the following proposition.
\begin{prop} Let $k \in \frac{1}{2}+\BZ$ with $k \ge \frac{5}{2}$, $f \in S^+_k(\Gamma_0(4)^{\star})$ a $L^2$-normalised Hecke eigenform. Assuming \eqref{eq:uniformbound} holds with $\beta>0$, then we have for $y \ge \frac{\sqrt{3}}{8}$:
$$\begin{aligned}
y^{\frac{k}{2}}|f(z)| & \ll_{\epsilon} \frac{k^{\frac{1}{4}+\frac{\alpha}{2}+\frac{\beta}{2}+\epsilon}}{y^{\frac{1}{2}+\frac{\beta}{2}}} \left( 1+yk^{-\frac{1}{2}}\right), \\
y^{\frac{k}{2}}|(f|_k W_4)(z)| & \ll_{\epsilon} \frac{k^{\frac{1}{4}+\frac{\alpha}{2}+\frac{\beta}{2}+\epsilon}}{y^{\frac{1}{2}+\frac{\beta}{2}}} \left( 1+yk^{-\frac{1}{2}}\right), \\
y^{\frac{k}{2}}|(f|_k V_4)(z)| & \ll_{\epsilon} \frac{k^{\frac{1}{4}+\frac{\alpha}{2}+\frac{\beta}{2}+\epsilon}}{y^{\frac{1}{2}+\frac{\beta}{2}}} \left( 1+yk^{-\frac{1}{2}}\right).
\end{aligned}$$
\label{prop:halfsmallab}
\end{prop}

If $y \ge \frac{3k}{\pi}$ (and $\beta\le k$) we can use the second part of Lemma \ref{lem:Sabest} with Lemma \ref{lem:expdecay} to get

$$\begin{aligned}
S \left(\frac{k-1+\beta}{2},2 \pi y,1 \right) & \ll (4 \pi y)^{-\frac{k}{2}-\frac{1}{2}-\frac{\beta}{2}} k^{\frac{k}{2}+\frac{\beta}{2}} e^{-\frac{k}{2}} \left(1+k^{\frac{1}{2}} e^{-\pi y} \right).
\end{aligned}$$
Thus we conclude the following proposition.
\begin{prop} Let $k \in \frac{1}{2}+\BZ$ with $k\ge \frac{5}{2}$ and, $f \in S^+(\Gamma_0(4)^{\star})$ a $L^2$-normalised Hecke eigenform. Assuming \eqref{eq:uniformbound} with $k\ge\beta>0$ then we have for $y \ge \frac{12k}{\pi}$:
$$\begin{aligned}
y^{\frac{k}{2}}|f(z)| & \ll_{\epsilon} \frac{k^{\frac{1}{4}+\frac{\alpha}{2}+\frac{\beta}{2}+\epsilon}}{y^{\frac{1}{2}+\frac{\beta}{2}}} \left( 1+k^{\frac{1}{2}}e^{-\pi y}\right), \\
y^{\frac{k}{2}}|(f|_k W_4)(z)| & \ll_{\epsilon} \frac{k^{\frac{1}{4}+\frac{\alpha}{2}+\frac{\beta}{2}+\epsilon}}{y^{\frac{1}{2}+\frac{\beta}{2}}} \left( 1+k^{\frac{1}{2}} e^{- \pi y}\right), \\
y^{\frac{k}{2}}|(f|_k V_4)(z)| & \ll_{\epsilon} \frac{k^{\frac{1}{4}+\frac{\alpha}{2}+\frac{\beta}{2}+\epsilon}}{y^{\frac{1}{2}+\frac{\beta}{2}}} \left( 1+k^{\frac{1}{2}}e^{-\frac{\pi y}{4}}\right).
\end{aligned}$$
\label{prop:halfbigab}
\end{prop}

If we assume the Lindel\"of hypothesis then the conjunction of Propositions \ref{prop:halfsmallab} and \ref{prop:halfbigab} with the observation \eqref{eq:reducing} gives Theorem \ref{thm:2}. If we use the unconditional result $(\alpha,\beta)=(\frac{1}{3}+\epsilon,\frac{1}{3}+\epsilon)$ instead we find that:
\begin{equation}
\max_{\xi \in \{I,W_4,V_4\}} \sup_{y \ge k^{\frac{1}{4}}} y^{\frac{k}{2}} |(f|_k \xi)(z)| \ll_{\epsilon} k^{\frac{1}{4}+\frac{1}{6}+\epsilon}.
\label{eq:nearcusp}
\end{equation}
The remaining region will be dealt with in the next section.

\subsection{Amplification}

We start by using the Bergman kernel as given in Theorem 4 of \cite{realweight} to deduce the identity

\begin{equation}
\sum_j \overline{f_j(w)}f_j(z)= \frac{3 (k-1)}{4 \pi} \sum_{\xi  \in \Gamma_0(4)^{\star}} \frac{1}{\left(\frac{z-\overline{w}}{2i}\right)^k} \Bigg |_k \xi,
\label{eq:plainkernel}
\end{equation}

where $|_k \xi$ is taken with respect to the variable $z$ and $\{f_j\}$ is an orthonormal basis of the whole space $S_k(\Gamma_0(4)^{\star})$. If we apply the Hecke operator $|_k T(m)$ to both sides with respect the variable $z$ we get
\begin{equation}
\sum_j \lambda_j(m)\overline{f_j(w)}f_j(z)= \frac{3 (k-1)}{4 \pi} m^{\frac{k}{2}-1}\sum_{\xi  \in \Gamma_0(4)^{\star} \xi_{1,m} \Gamma_0(4)^{\star}} \frac{1}{\left(\frac{z-\overline{w}}{2i}\right)^k} \Bigg |_k \xi,
\label{eq:heckekernel}
\end{equation}
with
$$
\xi_{1,m}= \left( \begin{pmatrix} 1 & 0 \\ 0 & m \end{pmatrix}, m^{\frac{1}{4}} \right).
$$
Let us denote with $A_j(m)=\lambda_j(m)m^{-\frac{k-1}{2}}$ the normalised Hecke eigenvalues. Further let $\CM$ be a finite set of squares of odd integers and $x_m$ arbitrary real numbers for $m \in \CM$. Using the identity 
$$
\lambda_j(m^2)\lambda_j(n^2)=\sum_{d|(m,n)} d^{k-1}\lambda_j\left( \frac{m^2n^2}{d^4} \right)
$$
we get the following equation
\begin{equation}\begin{aligned}
 &\sum_{j} \left | \sum_{m \in \CM} x_m A_j(m) \right |^2 \overline{f_j(w)}f_j(z) \\ =& \sum_{m_1,m_2 \in \CM} x_{m_1}x_{m_2} (m_1m_2)^{-\frac{k-1}{2}} \sum_j \lambda_j(m_1) \lambda_j(m_2) \overline{f_j(w)}f_j(z) \\
=& \sum_l y_l l^{-\frac{k-1}{2}} \sum_j \lambda_j(l) \overline{f_j(w)}f_j(z) \\
=& \frac{3 (k-1)}{4 \pi} \sum_l y_l l^{-\frac{1}{2}}  \sum_{\xi  \in \Gamma_0(4)^{\star} \xi_{1,l} \Gamma_0(4)^{\star}} \frac{1}{\left(\frac{z-\overline{w}}{2i}\right)^k} \Bigg |_k \xi,
\label{eq:amplified}
\end{aligned}\end{equation}
where
$$
y_l= \sum_{\substack{m_1,m_2 \in \CM, \\ d^2|(m_1,m_2), \\l=\frac{m_1m_2}{d^{4}}}} x_{m_1}x_{m_2}.
$$
Specialising to $w=z$ we get the inequality we are interested in:
\begin{equation}
\left |\sum_{m \in \CM} x_m A_1(m) \right|^2 \cdot y^k |f_1(z)|^2 \le \frac{3(k-1)}{4 \pi} \sum_l |y_l| l^{-\frac{1}{2}} \sum_{\gamma \in G_l(4)} d_{\gamma}(z)^{-k},
\label{eq:domain1}
\end{equation}
where
$$
d_{\gamma}(z)= \frac{|\gamma z-\overline{z}|\cdot|j(\gamma,z)|}{2yl^{\frac{1}{2}}}
$$
and
$$
G_l(4)= \{\gamma \in \GL2(\BZ) | \det \gamma =l \text{ and } \gamma \equiv \begin{pmatrix} \star & \star \\ 0 & \star \end{pmatrix} \mod(4)\}.
$$
Note that
$$
d_{\gamma}(z)^2=u(\gamma z, z)+1,
$$
where
$$
u(z,w)=\frac{|z-w|^2}{4 \im z \im w}.
$$
Now we want the same inequality with $f_1$ replaced with $f_1|_k W_4$ and $f_1|_k V_4$. For this we replace \eqref{eq:plainkernel} with the following
\begin{equation}
\sum_j \overline{(f_j|_k B)(w)}(f_j|_k B)(z)= \frac{3 (k-1)}{4 \pi} \sum_{\xi  \in B^{-1}\Gamma_0(4)^{\star}B} \frac{1}{\left(\frac{z-\overline{w}}{2i}\right)^k} \Bigg |_k \xi,
\label{eq:twistplainkernel}
\end{equation}
where $B \in \{W_4,V_4\}$. Now we apply $|_k B^{-1}T(m) B$ to both sides and proceed as before leading to
\begin{equation}
\left |\sum_{m \in \CM} x_m A_1(m) \right|^2 \cdot y^k |(f_1|_k B)(z)|^2 \le \frac{3(k-1)}{4 \pi} \sum_l |y_l| l^{-\frac{1}{2}} \sum_{\gamma \in B^{-1}G_l(4)B} d_{\gamma}(z)^{-k}.
\label{eq:domain2}
\end{equation}
Now we just have to note that both $W_4$ and $V_4$ stabilze $G_l(4)$ for odd $l$.\\

We now condsider two sets $\CM_1,\CM_2$ given by
$$\begin{aligned}
\CM_1 &= \{p^2 | \Lambda \le p < 2 \Lambda, p \neq 2 \},\\
\CM_2 &= \{p^4 | \Lambda \le p < 2 \Lambda, p \neq 2 \},
\end{aligned}$$
with
$$
x_m=\sign(A_1(m)), \forall m \in \CM_1 \text{ (respectively $\CM_2$)},
$$
for which we have
\begin{equation}\begin{aligned}
|y_l| &\ll \begin{cases} \Lambda, & l=1,\\ 1, & l=p^2q^2 \text{ with } p^2,q^2 \in \CM_1, \\ 0, & \text{otherwise},\end{cases} \\
|y_l| &\ll \begin{cases} \Lambda, & l=1,\\ 1, & l=p^4, p^4q^4 \text{ with } p^4,q^4 \in \CM_2, \\ 0, & \text{otherwise},\end{cases}
\label{eq:ybound}
\end{aligned}\end{equation}
respectively. We add now the two equations \eqref{eq:domain1} for $\CM=\CM_1,\CM_2$ and by Cauchy-Schwarz we see that the left hand side has a lower bound of
$$
\left( \sum_{m \in \CM_1\cup \CM_2} |A_1(m)|  \right)^2 y^k|f_1(z)|^2 \gg_{\epsilon} \Lambda^{2-\epsilon} y^k|f_1(z)|^2
$$
as $\max(|A_1(p^2)|,|A_1(p^4)|)\gg 1$. We get the same inequality also for the other cusps and conclude
\begin{equation}
\Lambda^{2-\epsilon} \max_{B \in \{I,W_4,V_4\}} y^k|(f_1|_kB)(z)|^2 \ll_{\epsilon} \sum_{\CM = \CM_1,\CM_2 } k\sum_l |y_l| l^{-\frac{1}{2}} \sum_{\gamma \in G_l(4)} (u(\gamma z,z)+1)^{-\frac{k}{2}}.
\label{eq:ampli}
\end{equation} Thus we are left to bound the right hand side. For this reason we define the following quantities:
\begin{equation}\begin{aligned}
M(z,l,\delta) &= |\{\gamma \in G_l(4) | u(\gamma z,z) \le \delta\}|, \\ 
M_{\star}(z,l,\delta) &= |\{\gamma \in G_l(4) | c\neq 0, (a+d)^2 \neq 4l \text{ and } u(\gamma z,z) \le \delta\}|, \\
M_{u}(z,l,\delta) &= |\{\gamma \in G_l(4) | c=0, a \neq d \text{ and } u(\gamma z,z) \le \delta\}|, \\
M_{p}(z,l,\delta) &= |\{\gamma \in G_l(4) | (a+d)^2 = 4l \text{ and } u(\gamma z,z) \le \delta\}|.
\label{eq:counting}
\end{aligned}\end{equation}

\begin{lem} For $z=x+iy \in \BH$ with $|x| \ll 1$ and $y \gg 1$ we have
\begin{equation}
\sum_{\substack{1 \le l \le L, \\ l \text{ is a square}}} M_{\star}(z,l,\delta) \ll_{\epsilon} \left(\frac{L^{\frac{1}{2}}}{y}+L\delta^{\frac{1}{2}}+L^{\frac{3}{2}}\delta\right)L^{\epsilon}.
\label{eq:Mstar}
\end{equation}
\label{lem:gen}
\end{lem}
\begin{proof} This is basically Lemma 4.1. of \cite{Thybrid}. The same proof carries through with ease as we don't care about a level aspect.
\end{proof}

\begin{lem} For $z=x+iy \in \BH$ with $|x| \ll 1$, $y \gg 1$ and $l \in \BN$ with $d(l)\ll 1$ we have
\begin{equation}
M_{u}(z,l,\delta) \ll 1+l^{\frac{1}{2}}\delta^{\frac{1}{2}}y.
\label{eq:Mu}
\end{equation}
\label{lem:upper}
\end{lem}

\begin{proof} This is part of the variant of Lemma 1.3 given in the appendix of \cite{IS95}
\end{proof}

\begin{lem} For $z=x+iy \in \BH$ with $|x|\ll 1$ and $y\gg 1$ we have

\begin{equation}
M_{p}(z,l,\delta) \ll 1+l^{\frac{1}{2}}\delta^{\frac{1}{2}}y.
\end{equation}
\label{lem:para}
\end{lem}

\begin{proof} This is Lemma 4.4 of \cite{Thybrid}. Although we don't restrict ourselves to such a fundamental domain, the same proof carries through.
\end{proof}

It is now not hard to bound the expression
$$
\sum_{\gamma \in G_l(4)} (1+u(\gamma z,z))^{-\frac{k}{2}}
$$
polynomially in $l,k,y$ for $k \ge \frac{5}{2}$, thus we omit the details. Instead we give the following insight. If $u(\gamma z,z) \ge k^{-1+\eta}$ for some positive real $\eta$, then the expression
$$
(1+u(\gamma z,z))^{-\frac{k}{2}}
$$
has super-polynomial decay in $k$, thus if $l,y$ only depend on $k$ polynomially we can completely neglect that part as follows:
\begin{equation}
\sum_{\gamma \in G_l(4)} (1+u(\gamma z,z))^{\frac{-k}{2}} \le \!\!\sum_{\substack{\gamma \in G_l(4),\\ u(\gamma z,z)\le k^{-1+\eta}}}\!\!1+(1+k^{-1+\eta})^{-\frac{k}{2}+\frac{5}{4}}\sum_{\gamma \in G_l(4)} (1+u(\gamma z,z))^{\frac{-5}{4}}.
\label{eq:superpoly}
\end{equation}
From now on we will assume, that $\Lambda$ and $y$ will depend polynomially on $k$, so that \eqref{eq:superpoly} becomes
\begin{equation}
\sum_{\gamma \in G_l(4)} (1+u(\gamma z,z))^{\frac{-k}{2}} \ll_{\eta} M(z,l,k^{-1+\eta}).
\label{eq:neglect}
\end{equation}

We now we use this inequality to estimate the right hand side of \eqref{eq:ampli}. We first consider the case $\CM=\CM_1$. The contribution of $l=1$ is
$$
k \Lambda \left(1+yk^{\frac{-1+\eta}{2}} \right)
$$
by Lemma \ref{lem:gen} with $L=1$ and Lemmata \ref{lem:upper}, \ref{lem:para}.
The contribution of $l>1$ is
$$
k\Lambda^{\epsilon}\left(\frac{1}{y}+\Lambda^2k^{\frac{-1+\eta}{2}}+\Lambda^4k^{-1+\eta} \right)
$$
for the generic matrices by Lemma \ref{lem:gen} with $L=2^4\Lambda^4$ and by Lemmata \ref{lem:upper} and \ref{lem:para} the contribution of the upper triangular and the parabolic matrices is
$$
k \sum_{l=p^2q^2} l^{-\frac{1}{2}} (1+l^{\frac{1}{2}}yk^{\frac{-1+\eta}{2}}) \ll_{\eta} k\left(1+\Lambda^2k^{\frac{-1+\eta}{2}}y\right).
$$
Thus we get that sum over $\CM_1$ is bounded by
\begin{equation}
k\Lambda^{\epsilon}\left(\Lambda+\Lambda^2yk^{\frac{-1+\eta}{2}}+\Lambda^4k^{-1+\eta}\right).
\label{eq:m1}
\end{equation}
For $\CM=\CM_2$ the contribution of $l=1$ is again
$$
k\Lambda \left(1+yk^{\frac{-1+\eta}{2}} \right).
$$
For $l>1$ the contribution of the generic matrices is
$$
k\Lambda^{\epsilon}\left(\frac{1}{y}+\Lambda^4k^{\frac{-1+\eta}{2}}+\Lambda^8k^{-1+\eta} \right)
$$
by Lemma \ref{lem:gen} with $L=2^8\Lambda^8$ and by Lemmata \ref{lem:upper} and \ref{lem:para} the contribution of the upper triangular and parabolic matrices is
$$
k \left(\sum_{l=p^4}+\sum_{l=p^4q^4}\right) l^{-\frac{1}{2}}  (1+l^{\frac{1}{2}}yk^{\frac{-1+\eta}{2}}) \ll_{\eta} k \left( \frac{1}{\Lambda}+\Lambda y k^{\frac{-1+\eta}{2}}+\frac{1}{\Lambda^2}+\Lambda^2yk^{\frac{-1+\eta}{2}} \right).
$$
Thus we get that the sum over $\CM_2$ is bounded by
\begin{equation}
k\Lambda^{\epsilon} \left(\Lambda+\Lambda^2yk^{\frac{-1+\eta}{2}} + \Lambda^4k^{\frac{-1+\eta}{2}}+\Lambda^8k^{-1+\eta} \right).
\label{eq:m2}
\end{equation}

Combining \eqref{eq:m1} and \eqref{eq:m2} with \eqref{eq:ampli} and letting $\epsilon,\eta$ be suitably small we get that
\begin{equation}
\max_{B \in \{I,W_4,V_4\}} y^k |f_1|_k B(z)|^2 \ll_{\epsilon} k^{1+\epsilon}\Lambda^{\epsilon}\left(\frac{1}{\Lambda}+ y k^{-\frac{1}{2}}+\Lambda^2k^{-\frac{1}{2}}+\Lambda^6k^{-1}\right).
\label{eq:sup}
\end{equation}
Since we may assume $y \le k^{\frac{1}{4}}$ by \eqref{eq:nearcusp} we can choose $\Lambda=k^{\frac{1}{7}}$ and we achieve
\begin{equation}
\sup_{z \in \BH} y^{\frac{k}{2}}|f_1(z)| \ll_{\epsilon} k^{\frac{3}{7}+\epsilon};
\label{eq:supi}
\end{equation}
completing the proof of Theorem \ref{thm:1}.

\subsection{Lower bounds} As in Theorem \ref{thm:3} let $\{f_j\} \subseteq S^+_k(\Gamma_0(4)^{\star})$ be an orthonormal basis of Hecke eigenforms of half-integral weight $k$ contained in the Kohnen plus space and let $\{F_j\} \subseteq S_{2k-1}(\SL2(\BZ))$ be the corresponding arithmetically normalised Hecke eigenforms under the Shimura map.\\

The first part of the first lower bound is trivial as
$$
\sup_{z \in \BH} y^{k} |f_j(z)|^2 \gg \langle f_j,f_j \rangle_{\Gamma_0(4)} = 1.
$$
The second part follows from the inequality
\begin{equation}\begin{aligned}
y^{\frac{k}{2}} |\widehat{f_j}(|D|)| &= \left |\int_0^1 y^{\frac{k}{2}} f(z) e(-|D|z) dx \right| \\
& \le e^{2 \pi |D| y} \cdot \int_0^1 y^{\frac{k}{2}}|f_j(z)| dx \\
& \le e^{2 \pi |D| y} \cdot \sup_{z' \in \BH} y'^{\frac{k}{2}} |f_j(z')|.
\label{eq:fourcoeflow}
\end{aligned} \end{equation}
This inequality in conjunction with the Propositions \ref{prop:coeffscritvalue}, \ref{prop:norm} and \ref{prop:sym} gives
$$
\sup_{z' \in \BH} y'^{\frac{k}{2}} |f_j(z')| \gg_{\epsilon} \frac{(4 \pi y)^{\frac{k}{2}}}{\Gamma(k)^{\frac{1}{2}}} \cdot |D|^{\frac{k-1}{2}}k^{-\epsilon}e^{-2 \pi |D|y} \cdot L\left(F_j,\legendre{D}{\cdot},\frac{1}{2}\right).
$$
The choice $y=\frac{k}{4 \pi |D|}$ gives the desired inequality. Similarly we have
\begin{equation}\begin{aligned}
\sup_{z' \in \BH} \sum_j y'^{k} |f_j(z')|^2 &\ge \sum_j \left(\int_0^1 dx \right) \left(\int_0^1 y^k |f_j(z)|^2 dx \right) \\
& \ge \sum_j \left(\int_0^1 y^{\frac{k}{2}}|f_j(z)|dx\right)^2 \\
& \ge  y^k e^{-4 \pi y} \cdot \sum_j \left|\widehat{f_j}(1)\right|^2.
\label{eq:lowerplus}
\end{aligned}\end{equation}
By Corollary \ref{cor:halfpluscoeff} we have
\begin{equation}\begin{aligned}
\sum_j \left|\widehat{f_j}(1)\right|^2 &= \frac{6 \cdot (4\pi)^{k-1}}{\Gamma(k-1)} \cdot \frac{2}{3} \left[1+(-1)^{ \left\lfloor \frac{k+\frac{1}{2}}{2} \right\rfloor} \pi \sqrt{2} \sum_{c \ge 1}H_c(1,1)J_{k-1}\left( \frac{\pi}{c} \right) \right] \\
& \gg \frac{(4\pi)^{k-1}}{\Gamma(k-1)} \cdot \left[1-2 \pi \sum_{c \ge 1}\left|J_{k-1}\left( \frac{\pi}{c} \right)\right| \right].
\label{eq:pluscoe}
\end{aligned}\end{equation}
Now we use the following proposition.
\begin{prop}\label{prop:JBesselverysmall}
One has for $\rho \ge 2 x^2$:
$$
|J_{\rho}(x)| \ll \frac{\left(\frac{x}{2} \right)^{\rho}}{\Gamma(\rho+1)}.
$$
\end{prop}
\begin{proof} See Proposition 8 of \cite{realweight}.
\end{proof}

If $k \ge 21$ we are able to apply it in order to get the estimate
$$
\sum_{c \ge 1}\left|J_{k-1}\left( \frac{\pi}{c} \right)\right| \ll \frac{\left(\frac{\pi}{2} \right)^{k-1}}{\Gamma(k)} \sum_{c \ge 1} \frac{1}{c^{k-1}} = o(1).
$$
Combining this with the equations \eqref{eq:lowerplus} and \eqref{eq:pluscoe} and making the choice $y=\frac{k}{4 \pi}$ gives the last lower bound.

\begin{acknowledgement}This work is based on my master's thesis, which I completed during November 2013 - April 2014 in Bristol, UK, and further improved upon during my PhD in Bristol, UK. I would like to thank Prof. Kowalski, for enabling me to do my master's thesis abroad, Dr. Saha, for his useful comments and discussions on the topic, and Prof. Harcos for his helpful remarks.
\end{acknowledgement}

\bibliography{Bibliography}

\begin{thebibliography}{Tem11}

\bibitem[DS13]{DasSeng}
Soumya Das and Jyoti Sengupta.
\newblock {$L^\infty$} norms of holomorphic modular forms in the case of
  compact quotient.
\newblock {\em {P}reprint, to appear in Forum Math.}, 2013.
\newblock {\tt arXiv:1301.3677}.

\bibitem[FJK13]{FJK}
Joshua~S Friedman, Jay Jorgenson, and Jurg Kramer.
\newblock Uniform sup-norm bounds on average for cusp forms of higher weights.
\newblock {\em {P}reprint}, 2013.
\newblock {\tt arXiv:1305.1348}.

\bibitem[HT12]{HT2}
Gergely Harcos and Nicolas Templier.
\newblock On the sup-norm of {M}aass cusp forms of large level: {II}.
\newblock {\em Int. Math. Res. Not. IMRN}, (20):4764--4774, 2012.

\bibitem[HT13]{HT3}
Gergely Harcos and Nicolas Templier.
\newblock On the sup-norm of {M}aass cusp forms of large level. {III}.
\newblock {\em Math. Ann.}, 356(1):209--216, 2013.

\bibitem[IS95]{IS95}
H.~Iwaniec and P.~Sarnak.
\newblock {$L^\infty$} norms of eigenfunctions of arithmetic surfaces.
\newblock {\em Ann. of Math. (2)}, 141(2):301--320, 1995.

\bibitem[Kir13]{halflevel}
Eren~Mehmet Kiral.
\newblock Bounds on sup-norms of half-integral weight modular forms.
\newblock {\em {P}reprint}, 2013.
\newblock {\tt arXiv:1309.7218}.

\bibitem[Koh80]{Ko80}
Winfried Kohnen.
\newblock Modular forms of half-integral weight on {$\Gamma _{0}(4)$}.
\newblock {\em Math. Ann.}, 248(3):249--266, 1980.

\bibitem[Koh85]{Ko85}
Winfried Kohnen.
\newblock Fourier coefficients of modular forms of half-integral weight.
\newblock {\em Math. Ann.}, 271(2):237--268, 1985.

\bibitem[KZ81]{KZ81}
W.~Kohnen and D.~Zagier.
\newblock Values of {$L$}-series of modular forms at the center of the critical
  strip.
\newblock {\em Invent. Math.}, 64(2):175--198, 1981.

\bibitem[Mic07]{ANTALfunc}
Philippe Michel.
\newblock Analytic number theory and families of automorphic {$L$}-functions.
\newblock In {\em Automorphic forms and applications}, volume~12 of {\em
  IAS/Park City Math. Ser.}, pages 181--295. Amer. Math. Soc., Providence, RI,
  2007.

\bibitem[Niw75]{MFohiwatIocTF}
Shinji Niwa.
\newblock Modular forms of half integral weight and the integral of certain
  theta-functions.
\newblock {\em Nagoya Math. J.}, 56:147--161, 1975.

\bibitem[Pet14]{Pet14}
Ian Petrow.
\newblock A twisted motohashi formula and weyl-subconvexity for $l$-functions
  of weight two cusp forms.
\newblock {\em {P}reprint}, 2014.
\newblock {\tt arXiv:1409.3524}.

\bibitem[Ran39]{ContrRamfunc}
R.~A. Rankin.
\newblock Contributions to the theory of {R}amanujan's function {$\tau(n)$} and
  similar arithmetical functions. {I}. {T}he zeros of the function
  {$\sum^\infty_{n=1}\tau(n)/n^s$} on the line {$ \Re s=13/2$}. {II}. {T}he
  order of the {F}ourier coefficients of integral modular forms.
\newblock {\em Proc. Cambridge Philos. Soc.}, 35:351--372, 1939.

\bibitem[Rud05]{R05}
Ze{\'e}v Rudnick.
\newblock On the asymptotic distribution of zeros of modular forms.
\newblock {\em Int. Math. Res. Not.}, (34):2059--2074, 2005.

\bibitem[Sah14]{Saha14}
Abhishek Saha.
\newblock On sup-norms of cusp forms of powerful level.
\newblock {\em {P}reprint}, 2014.
\newblock {\tt arXiv:1404.3179}.

\bibitem[Shi73]{MFhiw}
Goro Shimura.
\newblock On modular forms of half integral weight.
\newblock {\em Ann. of Math. (2)}, 97:440--481, 1973.

\bibitem[Ste14]{realweight}
Raphael~S. Steiner.
\newblock Uniform bounds on sup-norms of holomorphic forms of real weight.
\newblock {\em {P}reprint}, 2014.
\newblock {\tt arXiv:1406.2918}.

\bibitem[Tem11]{Thybrid}
Nicolas Templier.
\newblock Hybrid sup-norm bounds for {H}ecke-{M}aass cusp forms.
\newblock {\em To appear J. Eur. Math. Soc}, 2011.

\bibitem[Xia07]{Supnormintweight}
Honggang Xia.
\newblock On {$L^\infty$} norms of holomorphic cusp forms.
\newblock {\em J. Number Theory}, 124(2):325--327, 2007.

\bibitem[You14]{Young14}
Matthew~P. Young.
\newblock Weyl-type hybrid subconvexity bounds for twisted l-functions and
  heegner points on shrinking sets.
\newblock {\em {P}reprint}, 2014.
\newblock {\tt arXiv:1405.5457}.

\end{thebibliography}
\end{document}